\documentclass{amsart}
\usepackage{amscd}
\usepackage{amsmath}
\usepackage{amssymb}
\usepackage{amsthm}
\begin{document}

\newtheorem{thm}{Theorem}
\newtheorem{prop}[thm]{Proposition}
\newtheorem{llama}[thm]{Lemma}
\newtheorem{sheep}[thm]{Corollary}
\newtheorem{deff}[thm]{Definition}
\newtheorem{example}[thm]{Example}

\newcommand{\sthat}{\hspace{.1cm}| \hspace{.1cm}}
\newcommand{\id}{\operatorname{id} }
\newcommand{\acl}{\operatorname{acl}}
\newcommand{\dcl}{\operatorname{dcl}}
\newcommand{\irr}{\operatorname{irr}}
\newcommand{\aut}{\operatorname{Aut}}
\newcommand{\fix}{\operatorname{Fix}}

\newcommand{\oo}{\mathcal{O}}
\newcommand{\mm}{\mathcal{M}}

\title{An invitation to model-theoretic Galois theory.}
\author{Alice Medvedev \and Ramin Takloo-Bighash}
\thanks{The second author's work is partially supported by NSF grant DMS-0701753 and a research grant from the University of Illinois at Chicago.}
\address{Department of Mathematics, Statistics, and Computer Science \\
University of Illinois at Chicago \\
322 Science and Engineering Offices (M/C 249) \\
851 S. Morgan Street \\
Chicago, IL 60607-7045}
\email{alice@math.uic.edu}

\address{Department of Mathematics, Statistics, and Computer Science \\
University of Illinois at Chicago \\
322 Science and Engineering Offices (M/C 249) \\
851 S. Morgan Street \\
Chicago, IL 60607-7045}
\email{rtakloo@math.uic.edu}

\begin{abstract}
 We carry out some of Galois' work in the setting of an arbitrary first-order theory $T$. We replace the ambient algebraically closed field by a large model $\mm$ of $T$, replace fields by definably closed subsets of $\mm$, assume that $T$ codes finite sets, and obtain the fundamental duality of Galois theory matching subgroups of the Galois group of $L$ over $F$ with intermediate extensions $F \leq K \leq L$. This exposition of a special case of \cite{poizat} has the advantage of requiring almost no background beyond familiarity with fields, polynomials, first-order formulae, and automorphisms.
\end{abstract}

\maketitle

\section{Introduction.}

 Two hundred years ago, \'{E}variste Galois contemplated symmetry groups of solutions of polynomial equations, and Galois theory was born.
   Thirty years ago, Saharon Shelah found it necessary to work with theories that \emph{eliminate imaginaries}; for an arbitrary theory, he constructed a canonical definitional expansion with this property in \cite{ClaThe}.
  Poizat immediately recognized the importance of a theory already having this property in its native language; indeed, he defined ``elimination of imaginaries'' in \cite{poizat}.
  It immediately became clear (see \cite{poizat}) that much of Galois
theory can be developed for an arbitrary first-order theory that eliminates imaginaries. This model-theoretic version of
Galois theory can be generalized beyond finite or even infinite algebraic extensions,
and this can in turn be useful in other algebraic settings such as the study of Galois groups
of polynomial differential equations (already begun in \cite{poizat}) and linear difference
equations. On a less applied note, it is possible to bring further ideas into the model-theoretic setting, as is done in \cite{cohoanand} for the relation between Galois cohomology and homogeneous spaces.

 Here we rewrite parts of Galois' work in the language of model
theory, a special case of \cite{poizat}. Like Galois, we only
treat finite extensions, while \cite{poizat} addresses
finitely-generated infinite extensions, including those generated
by a solution of a differential equation.
 A nice exposition of the more general theory, as
well as all the model-theoretic prerequisites, can be found in
\cite{poimth}. This paper is the result of collaboration between a
number theorist who wanted to learn model theory, and a logician
who wanted to remember Galois theory. As such, it is entirely
elementary in both algebra and logic, and should be accessible to
anyone with any undergraduate background in both. It can also
motivate an algebraist to learn a little bit of logic, or
enlighten a logician about a bit of algebra.
 \texttt{new} In the interest of this accessibility, we do not give a proper treatment to elimination of imaginaries, and do not describe the construction, for a fixed theory $T$, of the definitional expansion $T^{eq}$ that eliminates imaginaries. That would require multisorted logic, which, though no harder than the usual one-sorted kind, is rarely taught in introductory courses. \texttt{endnew}
 Before we launch into the details, let us say which parts of Galois theory we replicate,
and which are lost.

 We see fields as definably-closed substructures of models of
the theory of algebraically closed fields, rather than as models
of the theory of fields. This is necessary because otherwise it
may be impossible to amalgamate several finite extensions into a
normal extension, and indeed it is not even clear what it would
mean for an extension to be normal. Thus not every first-order
theory can play the role of the theory of fields. In this paper, we define the basic Galois-theoretic notions for an arbitrary theory in place of the theory of
algebraically closed fields. However, the fundamental duality of Galois theory requires a harmless technical condition, namely coding finite sets (a special case of elimination of imaginaries).
 We fix an arbitrary theory $T$; where
Galois worked inside the complex numbers, we work inside
a sufficiently saturated model $\mm$ of $T$, as is usual in modern
model theory. Thus, all of our models are elementary submodels of
$\mm$ and all of our sets are subsets of $\mm$.

 Here are the parts of Galois theory that we replicate.
 We define the \emph{degree} of an extension, the
\emph{automorphism group} of an extension, and \emph{splitting}
and \emph{normal} extensions. For a normal extension $F \leq L$,
we prove the fundamental duality between intermediate extensions
$F \leq K \leq L$ and subgroups of the Galois group of $L$ over
$F$. Next, we show that $K$ is normal over $F$ if and only if the
corresponding subgroup is normal. In this case the Galois group of
$K$ over $F$ is the appropriate quotient group.

Here are the parts of Galois theory that one cannot hope to have
in this generality.
 Our language may have no function symbols, and our theory
may not admit quantifier elimination, so we must replace
polynomials by arbitrary first-order formulae. Similarly, in our
setting an extension $L$ of $K$ need not be a $K$-vector space.
Indeed, there is no reason to have a definable bijection between
$L$ and some cartesian power of $K$. This rules out the
possibility of defining the degree of an extension in terms of the
dimension of a vector space, and precludes norms, traces,
characters, and discriminants. The conceptual interpretation of the solvability of the Galois group is also lost, as there is no analog of radicals, and no conceptual characterization of cyclic extensions.

 Let us now actually define all these words and prove these statements.

\section{The Fundamental Duality.}

 Notation:\begin{itemize}
\item We work in a $\kappa$-saturated model $\mm$ of $T$; we say
that a subset $A \subset \mm$ is \emph{small} if $|A| < \kappa$.
\item  Given an $L$-formula $\phi(x,y)$ and a tuple $a \in \mm$, we say that another tuple $b \in \mm$ is a \emph{solution} of $\phi(a,y)$ if $\mm \models \phi(a,b)$.
\item  Given $A \subset \mm$, $L(A)$ is the language $L$ augmented by new constant symbols, one for each element of $A$; we naturally expand $\mm$ to an $L(A)$-structure by interpreting the new constant symbols as the corresponding elements of $A$.
\item  For a substructure $B \leq \mm$ and a subset $A \subset B$, we denote by $\aut(B/A)$ the group of \emph{partial elementary} maps from $B$ to $B$ fixing $A$ pointwise. (A partial elementary map preserves \emph{all} first-order properties, unlike a partial isomorphism, which only preserves atomic formulae.)
\item Unless otherwise specified, letters may denote finite tuples. Thus $a \in A$ should be read as ``$a$ is a tuple of elements of $A$.''
\end{itemize}

\begin{deff}
 Given a small $A \subset \mm$ and a tuple $b \in \mm$, we say that $b \in \acl(A)$ ($b$ is algebraic over $A$, or $b$ is in the algebraic closure of $A$) if there is an $L(A)$-formula $\phi(y)$ such that $b$ is one of finitely many solutions of $\phi(y)$. If $b$ is the \emph{only} solution of $\phi(y)$, we say that $b \in \dcl(A)$ ($b$ is definable over $A$, or $b$ is in the definable closure of $A$).\\
 If in addition there is no $L(A)$-formula $\psi(y)$ such that $\psi(y)$ still has $b$ as a solution and has fewer solutions than $\phi(y)$, we call $\phi(y)$ an \emph{irreducible formula of $b$ over $A$}, denoted $\irr(b/A)$.\\
 Given a small $A \subset \mm$ and a tuple $b \in \mm$, we define $\oo(b/A)$ to be the orbit of $b$ under $\aut(\mm/A)$.\\
  Given $b \in \acl(A)$, we define \emph{the degree of $b$ over $A$} to be
 $\deg(b/A) := | \oo(b/A) |$.
\end{deff}

 Clearly, $\irr(b/A)$ exists for any $b \in \acl(A)$; although many formulae may fit this definition, they all have the
same solution set; so we often abuse notation and speak of \emph{the} formula $\irr(b/A)$. Note also $\irr(b/A)$ is equivalent to $\irr(b/\dcl(A))$. It is easy to check that $\acl$ and $\dcl$ are indeed closure operators on subsets of $\mm$. It is well-known (and easy to show) that $b \in \acl(A)$ if and only if $\oo(b/A)$ is finite, and in that case $\oo(b/A)$ is the solution set of $\irr(b/A)$, so the degree of $b$ over $A$ is the number of solutions of $\irr(b/A)$. For fields in characteristic zero, this degree is precisely the degree of the usual Galois theory, while in positive characteristic it is the separable degree. It is also clear that the degree is preserved under interdefinability over $A$, that is $\deg(c/A) = \deg(b/A)$ for any tuple $c \in \dcl(Ab)$ such that $b \in \dcl(Ac)$, which allows us to define the degree of a finite extension.

\begin{deff}
 Given $A \subset B \subset \mm$, we say that $B$ is a \emph{finite extension of $A$} if there is a tuple $b$ of elements of $B \cap \acl(A)$ such that $B \subset \dcl(Ab)$; we say that $b$ \emph{generates $B$ over $A$}; we define the \emph{degree of $B$ over $A$} to be $\deg(B/A) := \deg(b/A)$.

 If in addition $\oo(c/A) \subset B$ for every tuple $c \in B$, we say that $B$ is a \emph{normal extension of $A$}. If there is some $b \in B$ such that $\oo(b/A) \subset B$ and $B \subset \dcl(A \cup \oo(b/A))$, we say that $B$ is the \emph{splitting extension of $\irr(b/A)$ over $A$}.
\end{deff}

\begin{llama}
 A definably closed splitting extension is normal.
\end{llama}
\begin{proof}
 Let $B = \dcl(B)$ be the splitting extension of $\irr(b/A)$ over $A$, that is, $B = \dcl(A \cup \oo(b/A))$. Now $B$ must be $\aut(\mm/A)$-invariant because $\oo(b/A)$ is. Therefore it contains any $\aut(\mm/A)$-orbit it intersects.
\end{proof}

\begin{llama}
 Degrees of finite extensions multiply in towers. That is, if $A \subset B \subset C$ are finite extensions, then $\deg(C/A) = \deg(C/B) \cdot \deg(B/A)$.
\end{llama}
\begin{proof}
 Let $b$ generate $B$ over $A$, and let $c$ generate $C$ over $B$. Clearly, the concatenation $bc$ generates $C$ over $A$. We need to show that
$|\oo(bc/A)| = |\oo(b/A)| \cdot |\oo(c/B)|$.
 For $d \in \oo(b/A)$, let $$S_d := \{ (d, \sigma(c)) \sthat
\sigma \in \aut(\mm/A) \mbox{ and } \sigma(b)=d \}$$
 Clearly $\oo(bc/A) = \cup_{d \in \oo(b/A)} S_d$ is a disjoint union of $\deg(b/A)$-many sets $S_d$. Since $S_b = \{ (b,\sigma(c)) \sthat \sigma \in \aut(\mm/Ab) \}$, it is the same size as $\oo(c/B)$. It suffices to show that $|S_d| = |S_b|$ for all $d \in \oo(b/A)$. This is true because the size of $S_d$ is a definable property of $d$: if $\phi(y,z)$ is the $L(A)$-formula such that $\phi(b,z) = \irr(c/Ab)$, then $b$ satisfies
$\psi(y) := \exists_{!n} z\, \phi(y,z)$, and so all $d \in \oo(b/A)$ must satisfy it too. If $\theta(y) = \irr(b/A)$, it is clear that $\irr(bc/A) = \theta(y) \wedge \phi(y,z)$.
\end{proof}

\begin{llama}
 If $B = \dcl(Ab)$ is a finite extension of $A$, then
$|\aut(B/A)| = |B \cap \oo(b/A)|$.
\end{llama}
\begin{proof}
 It suffices to construct a bijection between the two sets of allegedly the same size.
 Let $f : \aut(B/A) \rightarrow (B \cap \oo(b/A))$ be defined by $f(\sigma) := \sigma(b)$.
If $f(\sigma) = f(\tau)$, then $\sigma(b) = \tau(b)$, and so
$\sigma \circ \tau^{-1}$ is identity on $\dcl(Ab) = B$. Thus $f$
is injective. Given some $b' \in B \cap \oo(b/A)$ let $\sigma \in
\aut(\mm/A)$ be such that $\sigma(b) = b'$. Now
$$\sigma(B) = \sigma(\dcl(Ab)) = \dcl(\sigma(Ab)) = \dcl(A\sigma(b)) = \dcl(Ab')\subset B$$
is a definably closed subset of $B$ containing $A$ of the same
degree over $A$ as $B$. Thus $\sigma(B) = B$ by the previous
lemma, so $\sigma|_B = f^{-1}(b')$ and $f$ is surjective.
\end{proof}

\begin{sheep}
For a finite extension $B \supset A$, $\deg(B/A) = |\aut(B/A)|$ if
and only if $B$ is a normal extension of $A$.
\end{sheep}
\begin{proof}
Let $b \in B \cap \acl{A}$ be such that $B \subset B' :=
\dcl(Ab)$. Suppose that $B$ is a normal extension of $A$. Then
$\deg(B/A) = \deg(B'/A) = |\oo(b/A)|$ by definition, and
$|\oo(b/A)| = |\aut(B'/A)|$ by the last lemma. It suffices to show
that the restriction map from $\aut(B'/A)$ to $\aut(B/A)$ is
well-defined and bijective. Since any automorphism of $B$ can be
extended to an automorphism of $B'$ (and indeed of all of $\mm$),
the restriction is surjective. Since $B$ is normal, it is clearly
invariant under automorphisms. Since $B' = \dcl(Ab)$, an
automorphism $\sigma \in \aut(B'/A)$ is completely determined by
$\sigma(b)$ and a fortiori by the restriction of $\sigma$ to $B$,
so that restriction is injective.

Note that any automorphism $\sigma \in \aut(B/A)$ is determined by
$\sigma(b)$, so
$$|\aut(B/A)| \leq |B \cap \oo(b/A)| \leq |\oo(b/A)|$$
 with the last inequality strict if $B$ is not a
normal extension of $A$.
\end{proof}

Note that if $A \subset B \subset C$ and $C$ is normal over $A$,
then $C$ is also normal over $B$, as orbits of $\aut(C/B)$ are
subsets of orbits of $\aut(C/A)$.

\begin{sheep}
 If $A \subset B \subset C$ and $B$ and $C$ are normal over $A$, then
$|\aut(C/A)| = |\aut(C/B)| \cdot |\aut(B/A)|$ and in fact
 $$ 0 \rightarrow \aut(C/B) \rightarrow \aut(C/A) \rightarrow \aut(B/A) \rightarrow 0$$
is exact, so $\aut(C/B)$ is a normal subgroup of $\aut(C/A)$.
\end{sheep}
\begin{proof}
 Naturally, $\aut(C/B) \subset \aut(C/A)$. Since $B$ is normal, it is $\aut(C/A)$-invariant, so restriction gives a surjective homomorphism $\aut(C/A) \rightarrow \aut(B/A)$ whose kernel clearly is $\aut(C/B)$.
\end{proof}

\begin{deff}
 Suppose that $C=\dcl(C)$ is a finite extension of $A=\dcl(A)$, and $G := \aut(C/A)$.
 If $H$ is a subgroup of $G$, we let
 $$\fix(H) := \{ c \in C \sthat \forall h \in H \, h(c) = c \}$$ be the set of elements of $C$ fixed pointwise by every element of $H$.

 If $A \subset B \subset C$, we let $\fix(B) := \{ h \in G \sthat \forall b\in B\, h(b)=b\}$ be the subgroup of $G$ of elements that fix $B$ pointwise.
\end{deff}

 Note that for any subgroup $H$, the set $\fix(H)$ is definably closed.

\begin{llama}
 If $A \subset B \subset C$ are definably closed, and $C$ is a normal extension of $A$, then $H := \aut(C/B)$ is normal in $G := \aut(C/A)$ if and only if $B$ is a normal extension of $A$.
\end{llama}
\begin{proof}
 The last corollary proves one direction, so we need to prove the other. Suppose that $B$ is not normal, that is there is some $b \in B$ with $\oo(b/A) \not\subset B$. Since
$B$ is definably closed, and $\oo(b/A)$ is $B$-definable (since it is $A$-definable), there must be at least two elements $c,d \in \oo(b/A)$ not in $B$, and we may further assume that $d \in \oo(c/B)$.
 We now find $h \in H$ and some $g \in G$ such that $g^{-1}hg \notin H$. We will take $h$ witnessing $d \in \oo(c/B)$, that is such that $h(c) = d$. We will take $g$ witnessing that $c \in \oo(b/A)$, that is such that $g(b) = c$. Now $h(g(b)) = h(c) = d$, and since $d \neq c$, $g^{-1} (d) \neq g^{-1}(c) = b$, so $g^{-1}hg(b) \neq b$, so $g^{-1}hg(b) \notin H = \aut(C/B)$, as wanted.
\end{proof}

We now define a special case of elimination of imaginaries that is necessary for the fundamental duality of Galois theory. Recall that $\mm$ is a monster model of our theory $T$.

\begin{deff}
 We say that $T$ \emph{codes finite sets of tuples} if for any $n \in \mathbb{N}$ and for any finite $F \subset \mm^n$ there is some tuple $b$ such that for any automorphism $\sigma \in \aut(\mm)$ we have $\sigma(b) = b$ if and only if $\sigma(F) = F$. We call $b$ the \emph{code} of $F$.
\end{deff}

 \texttt{new}: Full elimination of imaginaries is equivalent to the same condition for all, not necessarily finite, definable (with parameters) sets $F$. For example from algebraic geometry, the code of a variety is (the finite tuple of generators of) its field of definition. \texttt{endnew}
  It is clear that the code $b$ of $F$ is well-defined up to interdefinability, and that if $F$ is $A$-definable, then $b \in \dcl(A)$.  We only use the coding of finite sets of tuples once, where the tuple is the generator of a finite extension.

 It is well-known that the theory of algebraically closed fields
codes finite sets of tuples: a set of $m$ $n$-tuples is coded by
the tuple of elementary multi-symmetric polynomials.
 Let $S$ be a set of $m$ $n$-tuples of variables; a polynomial $P$ in all these variables is multi-symmetric in $S$ if $P$ is invariant under permutations of $S$; $P$ is elementary multi-symmetric if it is monic, homogeneous, of degree at most $m$ in each variable.
 For example,
$(a+b+c, ab+ac+bc, abc)$ is a code of $\{a,b,c\}$ and $(a+c, ac ,
b+d, bd, ad+bc)$ is a code for $\{ (a,b), (c,d) \}$. A complete
proof would be neither short nor beautiful\footnote{``Dans la troisi\`{e}me section, il faut donc prouver cette \'{e}limination des imaginaires pour les corps alg\'{e}briquement clos; l'argument essential est un exercice sur les fonctions sym\'{e}triques (Lemme 5) dont l'auteur, sans doute par manque de culture, n'a pas trouv\'{e} de traces dans la litt\'{e}rature; il doit s'agir d'un r\'{e}sultat "bien connu", un de ceux dont on n'ose publier une d\'{e}monstration qu'en cas d'absolue n\'{e}cessit\'{e}.'' \cite{poizat} }.

 Here is an example of a theory that fails to code finite sets. In it, there is no Galois correspondence between subgroups of $\aut(C/A)$ and intermediate definably closed sets $B$ with $A\subset B \subset C$.

\begin{example}
Take a language $L$ with two binary relations $R$ and $S$, and let
$\mm$ be an $L$-structure with an infinite universe containing
points $a$, $b$, $c$, and $d$. Let $R^\mm := \{ (a,b), (b,a),
(c,d), (d,c) \}$ and $S^\mm := \{ (a,c), (c,a), (b,d), (d,b) \}$.
 Note that this theory does not code finite sets; in
particular, the unordered sets $\{a, b\}$ and $\{c, d\}$ are not
interdefinable with any tuples. Let $A$ be anything disjoint from
$\{a,b,c,d\}$; then $C := \{ a,b,c,d \} \cup A$ is a finite
extension of $A$. Then $\aut(C/A)$ contains four elements: an
automorphism may fix all four points, or it may switch two
disjoint pairs of them, for example switching $a$ with $d$ and $b$
with $c$. However, there are no definably closed $B$ with $A
\subset B \subset C$ except for $A$ and $C$.
\end{example}

That is the only obstruction.

\begin{thm}
\label{thethm}
 Suppose that $T$ codes finite sets, that $C=\dcl(C)$ is a normal extension of $A=\dcl(A)$, and $G := \aut(C/A)$. Then there is a bijection between subgroups of $G$ and intermediate definably closed extensions given by associating a subgroup $H$ to the set $\fix(H)$, and a set $B$ to the subgroup $\fix(B)$.
\end{thm}

\begin{proof}
 We need to show that $\fix(\fix(H)) = H$ for any $H$, and that $\fix(\fix(B)) = B$ for any $B$. The proof relies on the fact that the restriction map
$\aut(\mm/A) \rightarrow \aut(C/A)$ is well-defined and surjective.

 The second part is easier. Clearly, it suffices to show that $\fix(\fix(B)) \subset \dcl(B)$. Suppose that $c \notin \dcl(B)$; then there is some automorphism $\sigma \in \aut(\mm/B)$ such that $\sigma(c) \neq c$. Abusing notation we also denote the restriction of $\sigma$ to $C$ by $\sigma \in \aut(C/A)$. Since $\sigma \in \fix(B)$ but $\sigma(c) \neq c$, this $\sigma$ witnesses that $c \notin \fix(\fix(B))$.

For the first part, it suffices to find some tuple $b$ of elements
of $\fix(H)$ such that $g(b) \neq b$ for any $g \notin H$. This
$b$ will be a code of the orbit of a generator of $C/A$ under $H$.
Let $c$ be such that $C = \dcl(Ac)$, let $F := \{ h(c) \sthat h\in
H \}$, and let $b$ be the code of $F$. Note that $H$ acts
(faithfully transitively) on $F$, so in particular $h(F) = F$ for
all $h \in H$, so the entries of $b$ are in $\fix(H)$. On the
other hand, take some $g \in G$ that is not in $H$. Note that
$g(c) = h(c)$ implies that $g=h$, so $g(c) \notin F$. But $c \in
F$, so $g(c) \in g(F)$. Thus, $g$ does not leave $F$ invariant,
and therefore does not fix $b$ the code of $F$.
\end{proof}

\section{Further developments}

Since Shelah's invention of imaginaries in 1978 and Poizat's
seminal paper \cite{poizat} in 1983, much more has been done with
automorphism groups in model theory. Already \cite{poizat} speaks
about the absolute Galois group of $A$ acting not only on the
elements of the algebraic closure $\bar{A}$  (which correspond to
the algebraic types over $\bar{A}$) but also on the whole Stone
space $S(\bar{A})$ of types over $\bar{A}$. Another recent paper \cite{CaFa} formulates precisely how much elimination of imaginaries is necessary and sufficient for the Galois correspondence in Theorem \ref{thethm}.

Sometimes, the Galois group appears as a definable \emph{binding
group} inside the model. This already occurs in the earliest
application of this abstract theory, to linear differential
equations in \cite{poizat}. When anything remotely like this
happens, it is extremely useful, for example allowing one to
extract a definable field out of a definable group action. See
\cite{poizatgroups} or \cite{bigPillay} for an introduction to
binding groups. In some sense, this gives a definable
representation of the Galois group.

Algebraists have not been idle either: Galois would hardly
recognize the Galois theory in modern algebra textbooks. Galois
Theory was fully developed by Weber \cite{Weber}, Steinitz
\cite{Steinitz}, and Artin \cite{Artin}. The Galois theory of
infinite extensions was initiated by Krull \cite{Krull}. In these
developments a central role is played by the simple observation
that any extension field of finite degree is a finite dimensional
vector space over the ground field. This opens the way to import
ideas and techniques from linear algebra. Of course none of this
happens on our level of generality. The notion of Krull topology,
however, carries over to our setting with little modification.
Weil \cite{Weil} invented universal domains, of which our monster
models are a natural generalization, and also introduced
\emph{fields of definition}, which have an exact analog in
\emph{canonical parameters} of definable sets. Here we should also
mention the developments growing out of the introduction of Galois
theory in the setting of commutative rings
\cite{commutative-rings}, as well as Rasala's Inseparable
Splitting Theory \cite{Rasala}, as notable progress in the
algebraic aspects of the subject.

Many of the connections of Galois theory to number theory and algebraic geometry are via homological algebra. For example, Galois cohomology, at least in the commutative setting, is now an important tool in algebraic number theory and class field theory (see e.g. \cite{Neukirch}). Cohomological methods combined with representation theoretic, analytic, and algebro-geometric techniques have produced astonishing results in number theory (e.g. \cite{Wiles} and \cite{Harris-Taylor}). Non-abelian Galois cohomology is considerably more difficult to handle, and for that reason has not found much popularity among the mathematical public; though, the non-commutative $H^1$ is now routinely used in questions of classification and forms (see e.g. \cite{Satake}). Giraud's book \cite{Giraud} contains a comprehensive study of general non-abelian cohomology. Some of these notions have been put into the language of model theory in \cite{cohoanand} and most recently in \cite{evanspastori}. Further exploration of the connections between model theory and higher non-abelian cohomology seems rather inevitable, as no land this accessible and this pristine can keep off intruders for long.

\bibliography{galbib}{}

\begin{thebibliography}{10}

\bibitem{Artin}
Emil Artin.
\newblock {\em Galois {T}heory}.
\newblock Notre Dame Mathematical Lectures, no. 2. University of Notre Dame,
  Notre Dame, Ind., 1942.
\newblock Edited and supplemented with a section on applications by Arthur N.
  Milgram.

\bibitem{CaFa}
Enrique Casanovas and Rafel Farr{\'e}.
\newblock Weak forms of elimination of imaginaries.
\newblock {\em Math. Log. Q.}, 50(2):126--140, 2004.

\bibitem{commutative-rings}
S.~U. Chase, D.~K. Harrison, and Alex Rosenberg.
\newblock Galois theory and {G}alois cohomology of commutative rings.
\newblock {\em Mem. Amer. Math. Soc. No.}, 52:15--33, 1965.

\bibitem{evanspastori}
David~M. Evans and Elisabetta Pastori.
\newblock Second cohomology groups and finite covers, 2009.

\bibitem{Giraud}
Jean Giraud.
\newblock {\em Cohomologie non ab\'elienne}.
\newblock Springer-Verlag, Berlin, 1971.
\newblock Die Grundlehren der mathematischen Wissenschaften, Band 179.

\bibitem{Harris-Taylor}
Michael Harris and Richard Taylor.
\newblock {\em The geometry and cohomology of some simple {S}himura varieties},
  volume 151 of {\em Annals of Mathematics Studies}.
\newblock Princeton University Press, Princeton, NJ, 2001.
\newblock With an appendix by Vladimir G. Berkovich.

\bibitem{Krull}
Wolfgang Krull.
\newblock Galoissche {T}heorie der unendlichen algebraischen {E}rweiterungen.
\newblock {\em Math. Ann.}, 100(1):687--698, 1928.

\bibitem{Neukirch}
J{\"u}rgen Neukirch.
\newblock {\em Algebraic number theory}, volume 322 of {\em Grundlehren der
  Mathematischen Wissenschaften [Fundamental Principles of Mathematical
  Sciences]}.
\newblock Springer-Verlag, Berlin, 1999.
\newblock Translated from the 1992 German original and with a note by Norbert
  Schappacher, With a foreword by G. Harder.

\bibitem{bigPillay}
Anand Pillay.
\newblock {\em Geometric stability theory}, volume~32 of {\em Oxford Logic
  Guides}.
\newblock The Clarendon Press Oxford University Press, New York, 1996.
\newblock Oxford Science Publications.

\bibitem{cohoanand}
Anand Pillay.
\newblock Remarks on {G}alois cohomology and definability.
\newblock {\em J. Symbolic Logic}, 62(2):487--492, 1997.

\bibitem{poizat}
Bruno Poizat.
\newblock Une th\'eorie de {G}alois imaginaire.
\newblock {\em J. Symbolic Logic}, 48(4):1151--1170 (1984), 1983.

\bibitem{poimth}
Bruno Poizat.
\newblock {\em A course in model theory}.
\newblock Universitext. Springer-Verlag, New York, 2000.
\newblock An introduction to contemporary mathematical logic, Translated from
  the French by Moses Klein and revised by the author.

\bibitem{poizatgroups}
Bruno Poizat.
\newblock {\em Stable groups}, volume~87 of {\em Mathematical Surveys and
  Monographs}.
\newblock American Mathematical Society, Providence, RI, 2001.
\newblock Translated from the 1987 French original by Moses Gabriel Klein.

\bibitem{Rasala}
Richard Rasala.
\newblock Inseparable splitting theory.
\newblock {\em Trans. Amer. Math. Soc.}, 162:411--448, 1971.

\bibitem{Satake}
I.~Satake.
\newblock {\em Classification theory of semi-simple algebraic groups}.
\newblock Marcel Dekker Inc., New York, 1971.
\newblock With an appendix by M. Sugiura, Notes prepared by Doris
  Schattschneider, Lecture Notes in Pure and Applied Mathematics, 3.

\bibitem{ClaThe}
Saharon Shelah.
\newblock {\em Classification theory and the number of nonisomorphic models},
  volume~92 of {\em Studies in Logic and the Foundations of Mathematics}.
\newblock North-Holland Publishing Co., Amsterdam, 1978.

\bibitem{Steinitz}
Ernst Steinitz.
\newblock {\em Algebraische {T}heorie der {K}\"orper}.
\newblock Chelsea Publishing Co., New York, N. Y., 1950.

\bibitem{Weber}
H.~Weber.
\newblock Die allgemeinen {G}rundlagen der {G}alois'schen {G}leichungstheorie.
\newblock {\em Math. Ann.}, 43(4):521--549, 1893.

\bibitem{Weil}
Andr{\'e} Weil.
\newblock {\em Foundations of {A}lgebraic {G}eometry}.
\newblock American Mathematical Society Colloquium Publications, vol. 29.
  American Mathematical Society, New York, 1946.

\bibitem{Wiles}
Andrew Wiles.
\newblock Modular elliptic curves and {F}ermat's last theorem.
\newblock {\em Ann. of Math. (2)}, 141(3):443--551, 1995.

\end{thebibliography}
\bibliographystyle{plain}
\end{document}